\documentclass[12pt,reqno,draft]{amsart}
\usepackage{enumerate, latexsym, amsmath, amsfonts, amssymb, amsthm, graphicx, color}
\def\pmod #1{\ ({\rm{mod}}\ #1)}
\def\Z{\Bbb Z}
\def\N{\Bbb N}

\def\Q{\Bbb Q}

\def\F{\Bbb F}

\def\bg{\bigg}
\def\({\bg(}
\def\){\bg)}

\def\sign{{\rm sign}}

\def\ve{\varepsilon}

\def\Ack{\medskip\noindent {\bf Acknowledgments}}
\theoremstyle{plain}
\newtheorem{theorem}{Theorem}

\newtheorem{lemma}{Lemma}

\theoremstyle{definition}

\theoremstyle{remark}

\makeatletter
\@namedef{subjclassname@2010}{%
  \textup{2010} Mathematics Subject Classification}
\makeatother
 \vspace{4mm}

\begin{document}
 \baselineskip=17pt
\hbox{} {}
\medskip
\title[Cubes in finite fields and related permutations]
{Cubes in finite fields and related permutations}
\date{}
\author[Hai-Liang Wu and Yue-Feng She] {Hai-Liang Wu and Yue-Feng She}

\thanks{
\newline \indent Supported by the National Natural Science
Foundation of China (Grant No. 11971222).}

\address {(Hai-Liang Wu) School of Science, Nanjing University of Posts and Telecommunications, Nanjing 210023, People's Republic of China}
\email{\tt whl.math@smail.nju.edu.cn}

\address {(Yue-Feng She) Department of Mathematics, Nanjing
University, Nanjing 210093, People's Republic of China}
\email{{\tt she.math@smail.nju.edu.cn}}

\begin{abstract}
Let $p=3n+1$ be a prime with $n\in\mathbb{N}=\{0,1,\cdots\}$, and let $g\in\mathbb{Z}$ be a primitive root modulo $p$. Let $0<a_1<\cdots<a_n<p$ be all the cubic residues modulo $p$ in the interval $(0,p)$. Then clearly the sequence
$$a_1\ {\rm mod}\ p,\ a_2\ {\rm mod}\ p,\cdots, a_n\ {\rm mod}\ p$$
is a permutation $s_p(g)$ of the sequence
$$g^3\ {\rm mod}\ p,\ g^6\ {\rm mod}\ p,\cdots, g^{3n}\ {\rm mod}\ p.$$
In this paper, we shall determine the sign of this permutation.
\end{abstract}
\maketitle
\section{Introduction}
\setcounter{lemma}{0}
\setcounter{theorem}{0}
\setcounter{corollary}{0}
\setcounter{remark}{0}
\setcounter{equation}{0}
\setcounter{conjecture}{0}
\setcounter{proposition}{0}

Investigating permutations over finite fields is an active topic in both number theory and finite fields. Using the Lagrange interpolation formula, clearly each permutation over a finite field is in fact induced by a permutation polynomial. For example, let $p$ be an odd prime, and let $a$ be an integer with $p\nmid a$. Then $x\ {\rm mod}\ p\mapsto ax\ {\rm mod}\ p\ (x=0,1,\cdots,p-1)$ is a permutation over the finite field $\F_p=\Z/p\Z$. Zolotarev \cite{Z} showed that the sign of this permutation is precisely the Legendre symbol $(\frac{a}{p})$. Later Lerch \cite{ML} extended this result to the ring of residue classes modulo an arbitrary positive integer. In 2015 Brunyate and Clark \cite{Brunyate} extended this result to the higher dimensional vector spaces over finite fields.

Recently, Sun \cite{S} studied permutations involving squares in finite fields. In fact, let $p=2m+1$ be an odd prime. Let $0<b_1<\cdots<b_m<p$ be all the quadratic residues modulo $p$ in the interval $(0,p)$. Then clearly the sequence
$$1^2\ {\rm mod}\ p,\ 2^2\ {\rm mod}\ p,\cdots,m^2\ {\rm mod}\ p$$
is a permutation $\sigma_p$ of the sequence
$$b_1\ {\rm mod}\ p,\ b_2\ {\rm mod}\ p,\cdots, b_m\ {\rm mod}\ p.$$
Let $\sign(\sigma_p)$ denote the sign of $\sigma_p$. Sun \cite[Theorem 1.4]{S} obtained that
$$\sign(\sigma_p)=
\begin{cases}
1   &\mbox{if}\ p \equiv 3\pmod 8,\\
(-1)^{\frac{h(-p)+1}{2}} &\mbox{if} \ p\equiv 7 \pmod 8,
\end{cases}
$$
where $h(-p)$ denotes the class number of $\Q(\sqrt{-p})$. Later Petrov and Sun \cite{Sun era} determined the sign of $\sigma_p$ in the case $p\equiv1\pmod4$.

Motivated by the above work, in this paper, we mainly consider permutations concerning cubes in $\F_p=\Z/p\Z$ ($p$ is an odd prime). The case $p\equiv 2\pmod 3$ is trivial. In fact, clearly in this case
$$\{x^3\ {\rm mod}\ p:\ x=0,1,\cdots,p-1\}=\Z/p\Z,$$
and hence $x\ {\rm mod}\ p\mapsto x^3\ {\rm mod}\ p\ (x=0,1\cdots,p-1)$ is a permutation $\tau_p$ over $\Z/p\Z$. The sign of $\tau_p$ is a direct consequence of Lerch's result \cite{ML} and we have $\sign(\tau_p)=(-1)^{\frac{p+1}{2}}$. Readers may see \cite[Theorem 1.2]{Wang-Wu} for details.

Now we consider the non-trivial case $p\equiv 1\pmod3$. Let $p=3n+1$ be a prime with $n\in\N$, and let $g\in\Z$ be a primitive root modulo $p$. Let $0<a_1<\cdots<a_n<p$ be all the cubic residues modulo $p$ in the interval $(0,p)$. Then clearly the sequence
$$a_1\ {\rm mod}\ p,\ a_2\ {\rm mod}\ p,\cdots, a_n\ {\rm mod}\ p$$
is a permutation $s_p(g)$ of the sequence
$$g^3\ {\rm mod}\ p,\ g^6\ {\rm mod}\ p,\cdots, g^{3n}\ {\rm mod}\ p.$$
In order to state our result, we first introduce some notations. Let
$$\mathcal{P}:=\{0<x<p:\ x\ \text{is a primitive root modulo $p$}\}.$$
It is known that $4p$ can be uniquely written as
\begin{equation}\label{Eq. A in the introduction}
4p=r^2+3s^2\ (r,s\in\Z)
\end{equation}
with $r\equiv1\pmod3$, $s\equiv0\pmod3$ and $3s\equiv (2g^n+1)r\pmod p$. Let $\omega=e^{2\pi i/3}$ be a primitive cubic root of unity. As $p$ splits in $\Z[\omega]$ and $\Z[\omega]$ is a PID, we can write $p=\pi\bar{\pi}$ for some primary prime element $\pi\in\Z[\omega]$ with $(\frac{g}{\pi})_3=\omega$, where $\bar{\pi}$ denotes the conjugation of $\pi$ and the symbol $(\frac{\cdot}{\pi})_3$ is the cubic residue symbol modulo $\pi$. Readers may refer to \cite[Chapter 9]{IR} for details.

We also define
$$\delta_p:=|\{0<x<p/4:\ x\ \text{is a cubic residue modulo $p$}\}|,$$
$$\alpha_p:=|\{0<x<p/2:\ x\ \text{is a $6$-th power residue modulo $p$}\}|,$$
and
$$\gamma_p:=\left|\left\{0<x<p/2:\ \(\frac{x}{p}\)=1\ \text{and}\ \(\frac{x}{\pi}\)_3=\omega^2\right\}\right|,$$
where $|S|$ denotes the cardinality of a set $S$

With the above notations, we now state our main result.
\begin{theorem}\label{Thm. A} Let $p=3n+1$ be a prime with $n\in\N$. Then we have

{\rm (i)} If $p\equiv1\pmod{12}$, then
$$|\{g\in\mathcal{P}:\ \sign(s_p(g))=1\}|=|\{g\in\mathcal{P}:\ \sign(s_p(g))=-1\}|.$$

{\rm (ii)} If $p\equiv 7\pmod{12}$, then $\sign(s_p(g))$ is independent on the choice of $g$. Also, we have
$$\sign(s_p(g))=(-1)^{\delta_p+(1+\alpha_p)(1+r)+(h(-p)+1-2\alpha_p)(2-r+3s)/4+s(1+\gamma_p)+(n-2)/4},$$
where $h(-p)$ is the class number of $\Q(\sqrt{-p})$.
\end{theorem}
We will prove Theorem \ref{Thm. A} in the next section.
\maketitle
\section{Proof of Theorem \ref{Thm. A}}
\setcounter{lemma}{0}
\setcounter{theorem}{0}
\setcounter{corollary}{0}
\setcounter{remark}{0}
\setcounter{equation}{0}
\setcounter{conjecture}{0}
We first introduce some notations. Let $p=3n+1$ be a prime with $n\in\N$, and let $g\in\Z$ be a primitive root modulo $p$. Let $\omega=e^{2\pi i/3}$ be a primitive cubic root of unity.

As $p$ splits in $\Z[\omega]$ and $\Z[\omega]$ is a PID, we can write $p=\pi\bar{\pi}$ for some primary prime element $\pi\in\Z[\omega]$ with $(\frac{g}{\pi})_3=\omega$, where $\bar{\pi}$ denotes the conjugation of $\pi$ and the symbol $(\frac{\cdot}{\pi})_3$ is the cubic residue symbol modulo $\pi$. Readers may refer to \cite[Chapter 9]{IR} for details. For convenience, we use the symbol $\mathfrak{p}$ to denote the prime ideal $\pi\Z[\omega]$. Recall that $4p$ can be uniquely written as
\begin{equation}\label{Eq. quadratic forms represents 4p}
4p=r^2+3s^2\ (r,s\in\Z)
\end{equation}
with $r\equiv1\pmod3$, $s\equiv0\pmod3$ and $3s\equiv (2g^n+1)r\pmod p$. We begin with the following result (cf. \cite[Corollary 10.6.2(c)]{BEW}).
\begin{lemma}\label{Lem. sum of two cubes}
For any $0<k<p$, let
$$N(k):=|\{(x,y):\ 0<x,y<p, y^3-x^3\equiv k\pmod p\}|.$$
Then with the above notations we have
$$N(k)=\begin{cases}
 p+r-8   &\mbox{if}\ (\frac{k}{\pi})_3=1,\\
(2p-r+3s-4)/2 &\mbox{if} \ (\frac{k}{\pi})_3=\omega,\\
(2p-r-3s-4)/2 &\mbox{if}\ (\frac{k}{\pi})_3=\omega^2.
\end{cases}$$
\end{lemma}
For any $0<k<p$ we define
$$r_k:=\left|\left\{(x,y):\ 0<x<y<p,y-x\equiv k\pmod p, \(\frac{x}{\pi}\)_3=\(\frac{y}{\pi}\)_3=1\right\}\right|.$$
We need the following result.
\begin{lemma}\label{Lem. A in the proof of Thm. A}
We have the following congruence:
$$\sum_{0<k<p/2}r_{p-k}\equiv\left|\left\{0<x<p/4:\ \(\frac{x}{\pi}\)_3=1\right\}\right| \pmod 2.$$
\end{lemma}
\begin{proof}
By definition $\sum_{0<k<p/2}r_{p-k}$ is clearly equal to
\begin{equation}\label{Eq. A in Lem. A}
\left|\left\{(x,y):\ 0<x<y<p,\ y-x>p/2,\ \(\frac{x}{\pi}\)_3=\(\frac{y}{\pi}\)_3=1\right\}\right|.
\end{equation}
Replacing $y$ by $p-y$, we obtain that (\ref{Eq. A in Lem. A}) is equal to
\begin{equation}\label{Eq. B in Lem. A}
\left|\left\{(x,y):\ 0<x,y<p,\ x+y<p/2,\ \(\frac{x}{\pi}\)_3=\(\frac{y}{\pi}\)_3=1\right\}\right|.
\end{equation}
By the symmetry we clearly have
$$\sum_{0<k<p/2}r_{p-k}\equiv\left|\left\{0<x<p/4:\ \(\frac{x}{\pi}\)_3=1\right\}\right| \pmod 2.$$
This completes the proof.
\end{proof}
Now we define the following sets:
\begin{align*}
A_1:&=\left\{0<x<p/2:\ \(\frac{x}{\pi}\)_3=1\right\},\\
A_{\omega}:&=\left\{0<x<p/2:\ \(\frac{x}{\pi}\)_3=\omega\right\},\\
A_{\omega^2}:&=\left\{0<x<p/2:\ \(\frac{x}{\pi}\)_3=\omega^2\right\}.\\
\end{align*}
We have the following result (Recall that $\mathfrak{p}$ is the prime ideal $\pi\Z[\omega]$).
\begin{lemma}\label{Lem. B in the proof of Thm. A} Let $p\equiv7\pmod{12}$ be a prime. Then we have

{\rm (i)} Recall that
$$\alpha_p:=|\{0<x<p/2:\ x\ \text{is a $6$-th power residue modulo $p$}\}|.$$
Then we have
$$\prod_{x\in A_1}x\equiv(-1)^{1+\alpha_p}\pmod p.$$

{\rm (ii)} Let
$$\beta_p:=|\{0<x<p/2:\ \(\frac{x}{p}\)=1\ \text{and}\ \(\frac{x}{\pi}\)_3=\omega\}|.$$
Then we have
$$\prod_{x\in A_{\omega}}x\equiv (-1)^{1+\beta_p}\omega^2\pmod {\mathfrak{p}}.$$

{\rm (iii)} Recall that
$$\gamma_p:=|\{0<x<p/2:\ \(\frac{x}{p}\)=1\ \text{and}\ \(\frac{x}{\pi}\)_3=\omega^2\}|.$$
Then we have
$$\prod_{x\in A_{\omega^2}}x\equiv (-1)^{1+\gamma_p}\omega\pmod {\mathfrak{p}}.$$
\end{lemma}
\begin{proof}
(i) One can verify the following polynomial congruence:
$$\prod_{0<x<p,(\frac{x}{\pi})_3=1}(T-x)\equiv T^{n}-1\pmod p.$$
Hence we have
$$(-1)^{n/2}\(\prod_{x\in A_1}x\)^2\equiv -1\pmod p.$$
As $p\equiv3\pmod4$, we have
$$\(\prod_{x\in A_1}x\)^2\equiv 1\pmod p.$$
Let $\alpha_p$ be as the above. Then it is clear that
$$\prod_{x\in A_1}x\equiv(-1)^{n/2-\alpha_p}\equiv(-1)^{1+\alpha_p}\pmod p.$$

(ii) As in (i), we also have
$$\prod_{0<x<p,(\frac{x}{\pi})_3=\omega}(T-x)\equiv T^n-\omega\pmod{\mathfrak{p}}.$$
Hence we obtain
$$\(\prod_{x\in A_{\omega}}x\)^2\equiv \omega\pmod{\mathfrak{p}}.$$
Noting that $\omega=(\omega^2)^2$ is a quadratic residue modulo $\mathfrak{p}$, by the definition of $\beta_p$ we have
$$\prod_{x\in A_{\omega}}x\equiv(-1)^{1+\beta_p}\omega^2\pmod{\mathfrak{p}}.$$

(iii) With essentially the same method used in (ii), one can verify (iii). This completes the proof.
\end{proof}

Let $\Phi_{p-1}(T)$ be the $(p-1)$-th cyclotomic polynomial, and let
$$P(T):=\prod_{1\le i<j\le n}(T^{3j}-T^{3i}).$$
Then we have the following result (cf. \cite[Lemma 2.5]{Wu-She}).
\begin{lemma}\label{Lem. cyclotomic polynomial}
Let $G(T)$ be an integral polynomial defined by
$$G(T)=\begin{cases}(-1)^{(n-2)/4}\cdot n^{n/2}&\mbox{if}\ p\equiv3\pmod4,
\\(-1)^{(n-4)/4}\cdot n^{n/2}\cdot T^{(p-1)/4}&\mbox{if}\ p\equiv1\pmod4.\end{cases}$$
Then $\Phi_{p-1}(T)\mid P(T)-G(T)$.
\end{lemma}
Now we are in a position to prove our main result.

\noindent{\bf Proof of Theorem \ref{Thm. A}.} It follows from definition that
$$\sign(s_p)\equiv\prod_{1\le i<j\le n}\frac{g^{3j}-g^{3i}}{a_j-a_i}\pmod{\mathfrak{p}}.$$
We first consider the numerator. Since $p$ totally splits in the cyclotomic field $\Q(e^{2\pi i/(p-1)})$, we obtain that $\Phi_{p-1}(T)$ {\rm mod} $p\Z[T]$ totally splits in $\Z/p\Z[T]$. Also, the set of all primitive $(p-1)$-th roots of unity map bijectively onto the set of all primitive $(p-1)$-th roots of unity in the finite field $\F_p=\Z/p\Z$. Hence we have
\begin{equation}\label{Eq. A in the proof of Thm. A}
\Phi_{p-1}(T)\equiv \prod_{x\in\mathcal{P}}(T-x)\pmod{p},
\end{equation}
where
$$\mathcal{P}:=\{0<x<p: x\ \text{is a primitive root modulo $p$}\}.$$
By Lemma \ref{Lem. cyclotomic polynomial} and (\ref{Eq. A in the proof of Thm. A}) we have
$$\prod_{1\le i<j\le n}(g^{3j}-g^{3i})=P(g)\equiv G(g)\pmod p,$$
i.e.,
\begin{equation}\label{Eq. numerator in the proof of Thm. A}
\prod_{1\le i<j\le n}(g^{3j}-g^{3i})\equiv \begin{cases}(-1)^{(n-2)/4}\cdot n^{n/2}\pmod p&\mbox{if}\ 4\mid p-3,
\\(-1)^{(n-4)/4}\cdot n^{n/2}\cdot g^{(p-1)/4}\pmod p&\mbox{if}\ 4\mid p-1.\end{cases}
\end{equation}
By (\ref{Eq. numerator in the proof of Thm. A}) for any $g'\in\mathcal{P}$ we find that
$$\prod_{1\le i<j\le n}\frac{g^{3j}-g^{3i}}{(g^{'})^{3j}-(g^{'})^{3i}}\equiv \begin{cases}
 (g/g')^{\frac{p-1}{4}}\pmod p &\mbox{if}\ 4\mid p-1,\\
1 \pmod p&\mbox{if} \ 4\mid p-3.
\end{cases}$$
This implies $\sign(s_p(g))\cdot\sign(s_p(g^{-1}))=-1$ if $p\equiv1\pmod4$ and hence in the case $p\equiv1\pmod4$ we have
$$|\{g\in\mathcal{P}:\ \sign(s_p(g))=1\}|=|\{g\in\mathcal{P}:\ \sign(s_p(g))=-1\}|.$$
Also, in the case $p\equiv3\pmod4$, it is clear that $\sign(s_p(g))$ is independent on the choice of $g$.

We now consider the denominator and assume $p\equiv3\pmod4$. Recall that
$$r_k=\left|\left\{(x,y):\ 0<x<y<p,y-x\equiv k\pmod p, \(\frac{x}{\pi}\)_3=\(\frac{y}{\pi}\)_3=1\right\}\right|.$$
It is clear that
\begin{align*}
\prod_{1\le i<j\le n}(a_j-a_i)\equiv \prod_{0<k<p}k^{r_k}&\equiv(-1)^{\sum_{0<k<p/2}r_{p-k}}\cdot\prod_{0<k<p/2}k^{r_k+r_{p-k}}\\
&\equiv(-1)^{\delta_p}\prod_{0<k<p/2}k^{r_k+r_{p-k}}\pmod {\mathfrak{p}},
\end{align*}
where
$$\delta_p=|\{0<x<p/4:\ x\ \text{is a cubic residue modulo $p$}\}|.$$
The last congruence follows from Lemma \ref{Lem. A in the proof of Thm. A}. By the definition of $r_k$ one can verify that for all $0<k<p$ we have
\begin{equation}\label{Eq. r(k)+r(p-k)}
r_k+r_{p-k}=N(k)/9,
\end{equation}
where $N(k)$ is defined in Lemma \ref{Lem. sum of two cubes}. In view of the above, we obtain that
$\prod_{1\le i<j\le n}(a_j-a_i)$ {\rm mod} $\mathfrak{p}$ is equal to
$$(-1)^{\delta_p}\prod_{x\in A_1}x^{\frac{p+r-8}{9}}\prod_{y\in A_{\omega}}y^{\frac{2p-r+3s-4}{18}}\prod_{z\in A_{\omega^2}}z^{\frac{2p-r-3s-4}{18}}\ \text{\rm mod}\  {\mathfrak{p}}.$$
By Lemma \ref{Lem. B in the proof of Thm. A} we have
\begin{equation*}
\prod_{x\in A_1}x^{\frac{p+r-8}{9}}\equiv(-1)^{(1+\alpha_p)(1+r)}\pmod {\mathfrak{p}},
\end{equation*}
\begin{equation*}
\prod_{y\in A_{\omega}}y^{\frac{2p-r+3s-4}{18}}\prod_{z\in A_{\omega^2}}z^{\frac{2p-r-3s-4}{18}}
\equiv (-1)^{(\beta_p+\gamma_p)(-r+3s)/2+(1+\gamma_p)s}\omega^{2s/3}\pmod{\mathfrak{p}}.
\end{equation*}
Note that
$$\alpha_p+\beta_p+\gamma_p=|\{0<x<p/2:\ x\ \text{is a quadratic residue modulo $p$}\}|.$$
By the class number formula of $\Q(\sqrt{-p})$ (cf. \cite[Theorem 4, p. 346]{BS}) we know that
$$|\{0<x<p/2:\ x\ \text{is a quadratic residue modulo $p$}\}|\equiv \frac{h(-p)+1}{2}\pmod 2,$$
where $h(-p)$ is the class number of $\Q(\sqrt{-p})$. By the above we obtain that
$\prod_{1\le i<j\le n}(a_j-a_i)$ {\rm mod} $\mathfrak{p}$ is equal to
\begin{equation}\label{Eq. B in the proof of Thm. A}
(-1)^{\delta_p+(1+\alpha_p)(1+r)+(h(-p)+1-2\alpha_p)(2-r+3s)/4+s(1+\gamma_p)}\omega^{2s/3}\
\text{\rm mod}\ \mathfrak{p}.
\end{equation}
By (\ref{Eq. numerator in the proof of Thm. A}) we have
\begin{equation}\label{Eq. C in the proof of Thm. A}
\prod_{1\le i<j\le n}(g^{3j}-g^{3i})\equiv (-1)^{(n-2)/4}\cdot n^{n/2}\pmod p.
\end{equation}
It is also known that $3$ is a cubic residue modulo $p$ if and only if the equation
$$4p=X^2+243Y^2$$
has integral solutions. With our notations in (\ref{Eq. quadratic forms represents 4p}), it is equivalent to $s\equiv 0\pmod 9$. We now divide the remaining proof into two cases.

{\bf Case I.} $3$ is not a cubic residue modulo $p$.

In this case, as
$$\sign(s_p)\equiv\prod_{1\le i<j\le n}\frac{g^{3j}-g^{3i}}{a_j-a_i}\equiv\pm1\pmod{\mathfrak{p}},$$
we must have $n^{\frac{n}{2}}\equiv\ve\omega^{2s/3}$ for some $\ve\in\{\pm1\}$. Hence
$$\ve\equiv n^{3n/2}\equiv\(\frac{-3}{p}\)\equiv1\pmod {\mathfrak{p}}.$$
Combining this with (\ref{Eq. B in the proof of Thm. A}) and (\ref{Eq. C in the proof of Thm. A}), we obtain
$$\sign(s_p(g))=
(-1)^{\delta_p+(1+\alpha_p)(1+r)+(h(-p)+1-2\alpha_p)(2-r+3s)/4+s(1+\gamma_p)+(n-2)/4}.$$

{\bf Case II.} $3$ is a cubic residue modulo $p$.

In this case we have
$n^{n/2}=\pm1$ in this case and hence $$n^{n/2}=n^{3n/2}\equiv \(\frac{-3}{p}\)=1\pmod{\mathfrak{p}}.$$
Combining this with (\ref{Eq. B in the proof of Thm. A}) and (\ref{Eq. C in the proof of Thm. A}), we also obtain
$$\sign(s_p(g))=(-1)^{\delta_p+(1+\alpha_p)(1+r)+(h(-p)+1-2\alpha_p)(2-r+3s)/4+s(1+\gamma_p)+(n-2)/4}.$$
This completes the proof.\qed

\Ack\ This research was supported by the National Natural Science Foundation of
China (Grant No. 11971222). The first author was also supported by NUPTSF (Grant No. NY220159).

\end{document}